\documentclass{amsart}

\usepackage[a4paper, total={6in, 9in}]{geometry}
\usepackage{amsmath}
\usepackage[hidelinks]{hyperref}
\usepackage{amsfonts}
\usepackage{amssymb}
\usepackage{cleveref}
\usepackage{enumitem}

\newtheorem{theorem}{Theorem}[section]
\newtheorem{corollary}[theorem]{Corollary}
\newtheorem{lemma}[theorem]{Lemma}
\newtheorem{proposition}[theorem]{Proposition}
\theoremstyle{definition}
\newtheorem{definition}[theorem]{Definition}
\newtheorem{example}[theorem]{Example}
\theoremstyle{remark}
\newtheorem{remark}[theorem]{Remark}
\numberwithin{equation}{section}

\begin{document}

\title[Bounded Dehn surgery slopes]{Boundedness of Dehn surgery slopes admitting hyperbolic $\mathrm{PSL}(2,\mathbb{R})$-representations for two-bridge knots}
\author{Shunjiang Jiang and Ran Tao}

\address{Neijiang Vocational \& Technical College,
Neijiang, Sichuan Province, China}
\email{shunjiangnitu@foxmail.com}

\address{School of Mathematical Sciences and V.C. \& V.R. Key Lab, Sichuan Normal University, Chengdu, Sichuan Province, China}
\email{rtao\_sicnu@163.com}

	\begin{abstract}
		We study Dehn fillings on two-bridge knots via non-abelian representations into $\mathrm{PSL}(2,\mathbb{R})$ whose meridian image is hyperbolic. For each fixed nontrivial two-bridge knot, we prove that the set of surgery slopes admitting such representations is bounded. Equivalently, Dehn fillings along slopes with sufficiently large absolute value admit no non-abelian $\mathrm{PSL}(2,\mathbb{R})$ representations with hyperbolic meridian image. The proof combines the Riley polynomial with Khoi's surgery-slope formula. On each admissible real algebraic branch, we express the meridian and longitude translation parameters as functions of the branch parameter and derive uniform endpoint estimates for their quotient. The resulting bound is effective in principle but not optimized. We also provide examples illustrating the parameter sets and endpoint behavior.

	\end{abstract}

	\maketitle

	\section{Introduction}
	Let $K=K(\alpha,\beta)$ be a two-bridge knot in Schubert normal form, where $\alpha>1$ is odd, $\gcd(\alpha,\beta)=1$, and $-\alpha<\beta<\alpha$. Riley's foundational work on non-abelian representations of two-bridge knot groups~\cite{Riley2,Riley} initiated a fertile line of research on their representation theory. More broadly, representation and character varieties have become central tools for relating algebraic properties of knot groups to topological properties of $3$-manifolds~\cite{CullerShalen}. For two-bridge knots, the explicit group presentations and peripheral structures make real $\mathrm{SL}(2,\mathbb{R})$- and $\mathrm{PSL}(2,\mathbb{R})$-representations particularly tractable under Dehn filling. This leads naturally to the problem of determining which surgery slopes admit non-abelian real representations of the fundamental groups of the corresponding filled manifolds.

	A concrete motivation comes from representation volume, particularly the Seifert volume, which is defined using Godbillon--Vey invariants of $\mathrm{PSL}(2,\mathbb{R})$ representations~\cite{Volume,SU}. In their seminal work on Seifert volume, Brooks and Goldman~\cite{Volume} showed that, for surgeries on the figure-eight knot, the relevant Seifert-volume contribution is supported only on slopes in a bounded interval, namely $(-4,4)$, and vanishes outside it. In their analysis, this bounded-slope behavior is tied to the fact that the relevant non-abelian real representations occur only in that range. This suggests asking whether an analogous bounded-slope phenomenon holds more generally for real representations arising from Dehn fillings.

	In this paper we answer this question for non-abelian $\mathrm{PSL}(2,\mathbb{R})$ representations whose meridian image is hyperbolic.

	\begin{theorem}\label{thm:main}
		Let $K=K(\alpha,\beta)$ be a nontrivial two-bridge knot. Consider the rational slopes $p/q$ with $q\neq0$ for which $S^3_{p/q}(K)$ admits a non-abelian $\mathrm{PSL}(2,\mathbb{R})$ representation with hyperbolic meridian image. This set of slopes, viewed as a subset of $\mathbb{Q}\subset\mathbb{R}$, is bounded. Equivalently, there exists a constant $C(K)>0$ such that, for every rational slope $p/q$ with $q\neq0$ and $|p/q|>C(K)$, the manifold $S^3_{p/q}(K)$ admits no such representation. The same constant $C(K)$ also rules out non-abelian $\mathrm{SL}(2,\mathbb{R})$ representations with hyperbolic meridian image.
	\end{theorem}

	In terms of Seifert volume, Theorem~\ref{thm:main} rules out one possible source of large-slope contributions, namely non-abelian $\mathrm{PSL}(2,\mathbb{R})$ representations with hyperbolic meridian image. Any remaining nonzero contribution for large slopes would therefore have to come from the non-hyperbolic part of the representation theory, which requires separate input.

	The result is also relevant to approaches to left-orderability via representations into the universal cover $\widetilde{\mathrm{PSL}}(2,\mathbb{R})$~\cite{Orderable}. In many such arguments, one constructs real $\mathrm{SL}(2,\mathbb{R})$ representations and studies whether the associated projective representations lift to $\widetilde{\mathrm{PSL}}(2,\mathbb{R})$. This approach appears, for example, in work on twist knots~\cite{TeragaitoTwist}, genus-one two-bridge knots~\cite{HakaTeraGenusOne}, double twist knots~\cite{Tran2,Tran}, and related families~\cite{Gao,Thakar2023}. Theorem~\ref{thm:main} therefore shows that, for sufficiently large surgery slopes, such an $\mathrm{SL}(2,\mathbb{R})$ input cannot be non-abelian with hyperbolic meridian image.

	Our approach is based on the Riley polynomial together with the $\mathrm{SU}(1,1)$ parameterization developed below. The Riley equation provides a real parameter on each admissible algebraic branch of representations. Along such a branch, the meridian and longitude translation parameters become two real functions. The Dehn-filling relation then implies that every surgery slope realized by such a representation is, up to sign, the quotient of these functions. The proof therefore reduces to proving that these quotient functions are uniformly bounded, including as one approaches the endpoints of each branch. This gives a bound that is effective in principle, although we do not compute or optimize it here.

	The paper is organized as follows. Section~2 reviews the Riley polynomial and the $\mathrm{SU}(1,1)$ parameterization, and then derives the slope formula used in the proof. Section~3 proves Theorem~\ref{thm:main}, presents two illustrative computations, and records consequences for $\widetilde{\mathrm{PSL}}(2,\mathbb{R})$-based methods and Seifert volume.

\section{Representation and slope setup}
	This section fixes the representation-theoretic notation and derives the slope function used in the proof.
	\subsection{Riley polynomial}
	Let $\nu K$ be an open tubular neighborhood of $K$. The exterior $S^3\setminus\nu K$ has fundamental group
	\[
	\pi_1(S^3\setminus\nu K)=\langle x,y\mid wx=yw\rangle.
	\]
	Here $ w=x^{\varepsilon _1}y^{\varepsilon _2}\cdots x^{\varepsilon _{\alpha -2}}y^{\varepsilon _{\alpha -1}} $, where $ \varepsilon _i=(-1)^{\left\lfloor \beta i/\alpha\right\rfloor} $ for $ i=1,2,\ldots,\alpha-1 $, and $\left\lfloor\cdot\right\rfloor$ denotes the floor function. The meridian is $ \mu=x $, and the Seifert longitude is $ \lambda=x^{-2\sigma} w^* w $, where $w^*=y^{\varepsilon _{\alpha -1}}x^{\varepsilon _{\alpha -2}}\cdots  y^{\varepsilon _2}x^{\varepsilon _1}$ and $ \sigma= \sum_{j=1}^{\alpha -1}{\varepsilon _j}$. Riley considered the following assignment of the generators.

	\[
	\rho \left( x \right) =\left( \begin{matrix}
		t&		1\\
		0&		1\\
	\end{matrix} \right),
	\qquad
	\rho \left( y \right) =\left( \begin{matrix}
		t&		0\\
		-ts&		1\\
	\end{matrix} \right),
	\qquad t,s\in \mathbb{C}.
	\]
	Write $ W=\rho ^{\varepsilon _1}\left( x \right)\rho ^{\varepsilon _2}\left( y \right)\cdots \rho ^{\varepsilon _{\alpha -2}}\left( x \right)\rho ^{\varepsilon _{\alpha -1}}\left( y \right)=\left( \begin{smallmatrix} w_{11}&w_{12}\\w_{21}&w_{22} \end{smallmatrix} \right)  $. With this notation, Riley's criterion is the following.
	\begin{theorem}[Riley, Theorem~1~\cite{Riley}]
		The assignment $ \rho $ defines a non-abelian representation of the knot group if and only if the pair $(t,s)$ satisfies the equation
		\[
		w_{11}+(1-t)w_{12}=0.
		\]
	\end{theorem}

	We will also use the unimodular normal form $ \rho \left( x \right) =\left( \begin{smallmatrix}
		t^{1/2}&		1\\
		0&	t^{-1/2}\\
	\end{smallmatrix} \right)  $ and $ \rho \left( y \right) =\left( \begin{smallmatrix}
		t^{1/2}&		0\\
		-s&	t^{-1/2}\\
	\end{smallmatrix} \right)  $. This form has the same Riley polynomial. Recall that $ \sigma=\sum_{j=1}^{\alpha -1}{\varepsilon _j} $ is even, since it is a sum of the even number $\alpha-1$ of signs $\pm1$. Riley's normalized polynomial is
	\[
	\varPhi(t,s)=t^{-\sigma/2} \left[w_{11}+(1-t)w_{12}\right].
	\]
	Riley's normalized polynomial satisfies the following symmetry.
	\begin{proposition}[Riley~\cite{Riley}, Proposition~1]\label{P1}
		$ \varPhi(t,s)=\varPhi(t^{-1},s) $.
	\end{proposition}
	The passage from this symmetry to the variable $t+t^{-1}$ uses the following elementary algebraic observation.
	\begin{lemma}\label{lem:symmetric-laurent}
		Let $ R $ be a commutative ring. If $ P(t)\in R[t,t^{-1}] $ satisfies $ P(t)=P(t^{-1}) $, then $ P(t) $ can be written as a polynomial in $ t+t^{-1} $ with coefficients in $ R $.
	\end{lemma}
	\begin{proof}
		Write $ P(t)=\sum_{m=-N}^N c_mt^m $. The equality $P(t)=P(t^{-1})$ in the Laurent polynomial ring implies $c_m=c_{-m}$ for every $m$. Hence
		\[
		P(t)=c_0+\sum_{n=1}^N c_n\left(t^n+t^{-n}\right).
		\]
		Define polynomials $ U_0(z)=2 $, $ U_1(z)=z $, and $ U_{n+1}(z)=zU_n(z)-U_{n-1}(z) $ for $ n\ge 1 $. Then $ U_n\left( t+t^{-1} \right)=t^n+t^{-n} $ for every $ n\ge 1 $. Thus $ P(t)=c_0+\sum_{n=1}^N c_nU_n\left( t+t^{-1} \right) $, as required.
	\end{proof}
	We shall also use the following result of Riley on repeated factors.
	\begin{lemma}[Riley~\cite{Riley}, Lemma~3]\label{P2}
		Let $ R_0= \mathbb{Z}\left[ t,t^{-1} \right]$. With Riley's second parameter $u$ renamed as $s$, and with his exponent equal to $\sigma/2$, the polynomial $ \varPhi(t,s) $ has no repeated factors in $ R_0[s]$ except for $ R_0$-units. Moreover, the parabolic specialization $ \varPhi(1,s) $ has no repeated factors in $ \mathbb{Z} [s]$.
	\end{lemma}
	The parabolic specialization will be used at the small-parameter boundary in the proof of the main theorem. We next turn to the real form and to the hyperbolic parameters used to define the slope function.
	\subsection{\texorpdfstring{The $\mathrm{SU}(1,1)$ parameterization}{The SU(1,1) parameterization}}
	Following Khoi~\cite{SU}, we work with $ \mathrm{SU}(1,1) $, a subgroup of $\mathrm{SL}(2,\mathbb{C})$ conjugate to $\mathrm{SL}(2,\mathbb{R})$ that admits a convenient parameterization for matrix calculations. In the unimodular normal form, $ \rho \left( x \right) =\left( \begin{smallmatrix}
		t^{1/2}&		1\\
		0&	t^{-1/2}\\
	\end{smallmatrix} \right)  $ and $ \rho \left( y \right) =\left( \begin{smallmatrix}
		t^{1/2}&		0\\
		-s&	t^{-1/2}\\
	\end{smallmatrix} \right) $. We use the notation
	\[
	\mathrm{SL}(2,\mathbb{R}) \xrightarrow{\psi}\mathrm{SU}(1,1) \xrightarrow{\varphi}\left( \gamma ,\omega \right)
	\]
		as follows. Let $ A=\left( \begin{smallmatrix}
			a&		b\\
			c&		d\\
		\end{smallmatrix} \right) \in \mathrm{SL}(2,\mathbb{R}) $, $ J=\left( \begin{smallmatrix}
		1&		-i\\
		1&		i\\
	\end{smallmatrix} \right)  $ and $
	\mathrm{SU}(1,1) =\left\{ \left( \begin{smallmatrix}
		\alpha&		\beta\\
		\overline{\beta }&		\overline{\alpha }\\
	\end{smallmatrix} \right) \left| \left| \alpha \right|^2-\left| \beta \right|^2=1 \right. \right\}
	$. The maps $ \psi $ and $ \varphi $ are defined as follows.

		\[\psi :A\longmapsto JAJ^{-1}  \quad  \text{by}    \quad  \left( \begin{matrix}
			a&		b\\
			c&		d\\
		\end{matrix} \right) \longmapsto \left( \begin{matrix}
			\frac{a+d+\left( b-c \right) i}{2}&		\frac{a-d-\left( b+c \right) i}{2}\\
			\frac{a-d+\left( b+c \right) i}{2}&		\frac{a+d-\left( b-c \right) i}{2}\\
		\end{matrix} \right),
	\]

	\[
	\varphi :\gamma =\frac{\overline{\beta }}{\alpha}, \omega =\arg(\alpha)\pmod{2\pi}.
	\]
	The map $\psi$ is a group isomorphism, and $\varphi$ gives coordinates on $\mathrm{SU}(1,1)$ with the multiplication rule stated below. Thus we may identify
	\[
	\mathrm{SU}(1,1) =\left\{ \left( \gamma ,\omega \right) \left| \left| \gamma \right|<1,-\pi \leqslant \omega <\pi \right. \right\}.\]
	The following multiplication rule will be used later. If
	$
	\left( \gamma ,\omega \right)  \left( \gamma ',\omega ' \right) =\left( \gamma '',\omega '' \right)
	$, then
	\[
	\gamma''=\frac{\gamma +\gamma'e^{-2i\omega}}{1+\overline{\gamma }\gamma'e^{-2i\omega}},
	\omega''=\omega +\omega'+\frac{1}{2i}\log \left\{ \frac{1+\overline{\gamma }\gamma'e^{-2i\omega}}{1+\gamma \overline{\gamma'}e^{2i\omega}} \right\}.
	\]
	In this formula, the logarithm is taken with a branch compatible with the chosen argument throughout the calculation. The equality for the $\omega$-coordinate is understood modulo $2\pi$; after computing it, we choose the representative in $[-\pi,\pi)$.

	We next record how these coordinates behave on hyperbolic elements.
	\begin{definition}
		A matrix $ M $ in $ \mathrm{SL}(2,\mathbb{R}) $ is hyperbolic if $ \left| \operatorname{tr}\left( M \right) \right|>2 $. Then $M$ is conjugate to a matrix of the form
		$\left( \begin{smallmatrix} r&0\\0&r^{-1} \end{smallmatrix} \right)$,
		where $ r\ne 0 $. An element of $\mathrm{PSL}(2,\mathbb{R})$ is hyperbolic if a lift to $\mathrm{SL}(2,\mathbb{R})$ is hyperbolic; this is well-defined because the two lifts differ only by sign.
	\end{definition}
	\begin{definition}
		Let $\Gamma$ be either $\pi_1(S^3\setminus\nu K)$ or the fundamental group of a Dehn filling of $K$, and let $\mu\in \Gamma$ be the image of the fixed knot meridian. A representation $\bar{\rho}:\Gamma\to \mathrm{PSL}(2,\mathbb{R})$ is called hyperbolic if $\bar{\rho}(\mu)$ is a hyperbolic element of $\mathrm{PSL}(2,\mathbb{R})$. Equivalently, any lift of $\bar{\rho}(\mu)$ to $\mathrm{SL}(2,\mathbb{R})$ has trace of absolute value greater than $2$.
	\end{definition}
		Let $ h=\varphi\circ\psi $. The following formulas follow immediately from $ \gamma =\frac{r-r^{-1}}{r+r^{-1}} $.
	\begin{equation}\label{eq:sl2r-parameterization}
		\begin{aligned}
			h\left( \left( \begin{smallmatrix}
				r&		0\\
				0&		r^{-1}\\
			\end{smallmatrix} \right) \right) & =\left( \tanh k,0 \right), \quad e^k=r,\ \text{if } r > 0, \\
			h \left( \left( \begin{smallmatrix}
				r&		0\\
				0&		r^{-1}\\
			\end{smallmatrix} \right) \right) & =\left( \tanh k,\pi \right), \quad e^k=-r,\ \text{if } r < 0.
		\end{aligned}
	\end{equation}
		Thus $k=\ln \left| r \right|$. In particular, $k>0$ when $\left| r \right|>1$, and $k<0$ when $\left| r \right|<1$. Moreover, taking powers has a particularly simple form: if
		$
		v=\left( \tanh c,n\pi \right)
		$
		with $n,\ell\in \mathbb{Z}$, then, with the angular coordinate understood modulo $2\pi$,
		\[
		v^\ell=\left( \tanh \ell c,\ell n\pi \right).
		\]
		Indeed, for angular coordinates in $\pi\mathbb{Z}$, the multiplication rule reduces in the first coordinate to the usual addition formula for $\tanh$.
	\subsection{Dehn surgery}
		We now express the Dehn-filling equation in terms of meridian and longitude translation parameters.

		We use the convention that $p/q$-surgery, with $q\neq0$, kills $\mu^p\lambda^q$. Thus the manifold obtained by Dehn surgery of finite slope $p/q$ on $K\subset S^3$ has fundamental group
		\[
		\pi_1\left(S^3_{p/q}(K)\right) = \langle x,y\mid wx=yw,\ \mu ^p\lambda ^q=1\rangle.
		\]
			Equivalently, the filling relation may be written as $\mu^p=\lambda^{-q}$. Suppose that $\bar{\rho}:\pi_1(S^3_{p/q}(K))\to\mathrm{PSL}(2,\mathbb{R})$ is a representation, and pull it back to the knot exterior $S^3\setminus\nu K$.
			We need an $\mathrm{SL}(2,\mathbb{R})$ lift of this exterior representation. The obstruction class is the pullback of the class of the central extension $1\to\{\pm I\}\to\mathrm{SL}(2,\mathbb{R})\to\mathrm{PSL}(2,\mathbb{R})\to1$.
			This obstruction vanishes because $H^2(S^3\setminus\nu K;\mathbb{Z}/2)=0$ for a knot exterior. After choosing a lift $\rho:\pi_1(S^3\setminus\nu K)\to\mathrm{SL}(2,\mathbb{R})$, the element $\mu^p\lambda^q$ has trivial image in $\mathrm{PSL}(2,\mathbb{R})$ and therefore satisfies $\rho(\mu^p\lambda^q)\in\{\pm I\}$.
			Thus the filling relation lifts to $\rho(\mu)^p\rho(\lambda)^q=\pm I$. If necessary, multiply the lift by the $\{\pm I\}$-valued character of the knot exterior that sends a meridian to $-I$; this does not change the projective representation and lets us choose the hyperbolic meridian lift to have positive eigenvalues. With this choice, Riley's normal form and Khoi's real $\mathrm{SU}(1,1)$ conditions apply to the lifted exterior representation~\cite[Theorem~1]{Riley}~\cite[Section~6]{SU}.

		To simplify the slope calculation, we conjugate the whole representation by a matrix that diagonalizes
		$ \rho \left( x \right) =\left(\begin{smallmatrix}
			\scriptstyle	t^{1/2}&	\scriptstyle	1\\
			\scriptstyle	0&	\scriptstyle t^{-1/2}\\
		\end{smallmatrix} \right)  $. For $t\neq \pm 1$, this is possible because $\rho(x)$ has the two distinct eigenvalues $t^{1/2}$ and $t^{-1/2}$. After this conjugation we may assume
		\[
		\rho(\mu)=\rho(x)=\left( \begin{matrix}
			t^{1/2}&0\\
			0&t^{-1/2}
		\end{matrix} \right).
		\]
		Since $\lambda$ commutes with $\mu$, the matrix $\rho(\lambda)$ preserves the one-dimensional eigenspaces of $\rho(\mu)$; hence, after the same conjugation, $\rho(\lambda)$ is diagonal as well.

		We may therefore arrange that
		\[
		\rho(\mu)=\left( \begin{matrix}
			t^{1/2}&		0\\
			0&		t^{-1/2}\\
		\end{matrix} \right),
		\qquad
		\rho(\lambda)=\left( \begin{matrix}
			\lambda _0&		0\\
			0&		\lambda _0^{-1}\\
		\end{matrix} \right),
		\]
		where $ \lambda_0 $ depends on $ t $ and $ s $. Using the parameterization~\eqref{eq:sl2r-parameterization}, write the signed diagonal parameters as
		\[
		h(\rho(\mu))=\left( \tanh \tilde a(s),0 \right),
		\qquad
		h(\rho(\lambda))=\left( \tanh \tilde b(s),\delta\pi\right),
		\]
		where $\delta=0$ or $1$ records the sign of $ \lambda_0 $. We fix the choices of signed parameters on each branch. Put $z=t+t^{-1}$, and write $z=z(s)$ on the chosen branch. With the above diagonal normalization, $t^{1/2}=e^{\tilde a(s)}$ and hence
		\begin{equation}
			f(s):=\frac{z(s)}{2}=\frac{t+t^{-1}}{2}=\cosh(2\tilde a(s)).
			\label{eq:f-definition}
		\end{equation}
		By Proposition~\ref{P1} and Lemma~\ref{lem:symmetric-laurent}, the Riley polynomial is a polynomial in $t+t^{-1}$ with coefficients in $\mathbb{Z}[s]$. Thus $z(s)$, and hence $f(s)$, is obtained from the Riley equation on the chosen algebraic branch. To express the longitude trace in the same variable, we use the following word-level symmetry.
		\begin{lemma}\label{lem:w-star-symmetry}
			Let $u=t^{1/2}$, and write
			\[
			x(u)=\left( \begin{smallmatrix}
				u&		1\\
				0&		u^{-1}
			\end{smallmatrix} \right),\qquad y(u)=\left( \begin{smallmatrix}
				u&		0\\
				-s&		u^{-1}
			\end{smallmatrix} \right),\qquad P=\left( \begin{smallmatrix}
				0&		1\\
				1&		0
			\end{smallmatrix} \right).
			\]
			Then for $\epsilon=\pm 1$ we have
			\[
			P\bigl(x(u)^{\epsilon}\bigr)^TP=x\left(u^{-1}\right)^{\epsilon},\qquad P\bigl(y(u)^{\epsilon}\bigr)^TP=y\left(u^{-1}\right)^{\epsilon}.
			\]
			Consequently,
			\[
			Pw(u)^TP=w^*\left(u^{-1}\right),\qquad P\bigl(w^*(u)\bigr)^TP=w\left(u^{-1}\right).
			\]
		\end{lemma}
		\begin{proof}
			Direct computation gives
			\[
			Px(u)^TP=\left( \begin{smallmatrix}
				u^{-1}&		1\\
				0&		u
			\end{smallmatrix} \right)=x\left(u^{-1}\right),\qquad Py(u)^TP=\left( \begin{smallmatrix}
				u^{-1}&		0\\
				-s&		u
			\end{smallmatrix} \right)=y\left(u^{-1}\right).
			\]
			The inverse cases follow by taking inverses of both sides, using $P^{-1}=P$ and $(Y^T)^{-1}=(Y^{-1})^T$ for any invertible matrix $Y$. Since transpose reverses the order of a product, applying these identities successively to the word
			\[
			w(u)=x(u)^{\varepsilon_1}y(u)^{\varepsilon_2}\cdots x(u)^{\varepsilon_{\alpha-2}}y(u)^{\varepsilon_{\alpha-1}}
			\]
			gives
			\[
			Pw(u)^TP=y\left(u^{-1}\right)^{\varepsilon_{\alpha-1}}x\left(u^{-1}\right)^{\varepsilon_{\alpha-2}}\cdots y\left(u^{-1}\right)^{\varepsilon_2}x\left(u^{-1}\right)^{\varepsilon_1}=w^*\left(u^{-1}\right).
			\]
			The second identity is proved in the same way.
		\end{proof}

		Applying this to the longitude gives the corresponding trace symmetry. We use the capital $\Lambda$ for the matrix image of the longitude word, reserving the lowercase $\lambda$ for the group element.
		\begin{proposition}\label{prop:lambda-trace-symmetry}
			Let $\Lambda(t,s)$ denote the longitude matrix in the Riley representation, after writing the unimodular parameter as $t=u^2$. Then $\operatorname{tr}\Lambda(t,s)$ lies in $\mathbb{Z}[s][t,t^{-1}]$ and satisfies
			$
			\operatorname{tr}\Lambda\left(t,s\right)=\operatorname{tr}\Lambda\left(t^{-1},s\right).
			$
		\end{proposition}
		\begin{proof}
			Let $u=t^{1/2}$, and write $\Lambda(u)$ for the longitude matrix after substituting $t=u^2$. First note that the trace contains only even powers of $u$. Indeed, each of $x(u)^{\pm1}$ and $y(u)^{\pm1}$ has diagonal entries with odd $u$-degree and off-diagonal entries with even $u$-degree, with zero entries ignored. By induction on word length, a product of an even number of such factors has diagonal entries involving only even powers of $u$. The longitude word $x^{-2\sigma}w^*w$ has even length after expanding powers, since $-2\sigma$ and $2(\alpha-1)$ are even. Hence $\operatorname{tr}\Lambda(u)$ belongs to $\mathbb{Z}[s][u^2,u^{-2}]$, and we may regard it as a Laurent polynomial $\operatorname{tr}\Lambda(t,s)$ in $t=u^2$.

			By Lemma~\ref{lem:w-star-symmetry},
			\[
			w^*\left(u^{-1}\right)=Pw(u)^TP,\qquad w\left(u^{-1}\right)=P\bigl(w^*(u)\bigr)^TP.
			\]
			The same conjugation-transpose identity is preserved under integral powers, because it is preserved under products and inverses. Hence
			\[
			x\left(u^{-1}\right)^{-2\sigma}=P\bigl(x(u)^{-2\sigma}\bigr)^TP.
			\]
			Therefore
			\begin{align*}
				\Lambda\left(u^{-1}\right)
				&=x\left(u^{-1}\right)^{-2\sigma}w^*\left(u^{-1}\right)w\left(u^{-1}\right)\\
				&=P\bigl(x(u)^{-2\sigma}\bigr)^TP\,Pw(u)^TP\,P\bigl(w^*(u)\bigr)^TP\\
				&=P\Bigl( x(u)^{-2\sigma}\Bigr)^Tw(u)^T\bigl(w^*(u)\bigr)^TP\\
				&=P\Bigl( w^*(u)w(u)x(u)^{-2\sigma}\Bigr)^TP.
			\end{align*}
			Taking traces and using invariance of trace under conjugation, transpose, and cyclic permutation, we get
			\[
			\operatorname{tr}\Lambda\left(u^{-1}\right)=\operatorname{tr}\Bigl( w^*(u)w(u)x(u)^{-2\sigma}\Bigr)=\operatorname{tr}\Bigl( x(u)^{-2\sigma}w^*(u)w(u)\Bigr)=\operatorname{tr}\Lambda(u).
			\]
			Since $u^2=t$, this is exactly
			\[
			\operatorname{tr}\Lambda\left(t,s\right)=\operatorname{tr}\Lambda\left(t^{-1},s\right).\qedhere
			\]
		\end{proof}

			Proposition~\ref{prop:lambda-trace-symmetry} and Lemma~\ref{lem:symmetric-laurent} show that $\operatorname{tr}\rho(\lambda)$ depends on $t$ only through $t+t^{-1}$. Hence, after choosing a branch of the Riley curve, the longitude trace is an algebraic function of the branch parameter $s$. We set
			\begin{equation}
			g(s)=\frac{|\operatorname{tr}\rho(\lambda)|}{2}=\frac{|\lambda_0+\lambda_0^{-1}|}{2}=\cosh \tilde b(s).
			\label{eq:g-definition}
			\end{equation}

			The absolute value in~\eqref{eq:g-definition} removes the possible central sign in the chosen $\mathrm{SL}(2,\mathbb{R})$ lift. Together with $f(s)$ from~\eqref{eq:f-definition}, it determines the nonnegative translation parameters
			\[
			A(s):=|\tilde a(s)|=\frac12\operatorname{arccosh} f(s),\qquad
			B(s):=|\tilde b(s)|=\operatorname{arccosh} g(s).
			\]
			We use the signed parameters only to read off the Dehn-filling relation.

			Indeed, if the $\mathrm{PSL}(2,\mathbb{R})$ representation extends over the $p/q$-filling, then the lifted exterior representation satisfies
			\[
			\rho(\mu)^p\rho(\lambda)^q=\pm I.
			\]
			Both central elements $\pm I$ have $\gamma=0$ in the $(\gamma,\omega)$-coordinates. Using the power formula above and the fact that the relevant angular coordinates lie in $\pi\mathbb{Z}$, the $\gamma$-coordinate of $\rho(\mu)^p\rho(\lambda)^q$ is $\tanh\bigl(p\tilde a(s)+q\tilde b(s)\bigr)$. Hence
			\[
			p\tilde a(s)=-q\tilde b(s).
			\]
			The angular congruences depending on the sign $\pm I$ are irrelevant for this necessary condition. Since $\tilde a(s)=0$ corresponds to a non-hyperbolic meridian, every hyperbolic representation extending over the $p/q$-filling with $q\neq0$ must satisfy
			\begin{equation}\label{slope}
			-\frac{p}{q}=\frac{\tilde b(s)}{\tilde a(s)}.
			\end{equation}

			Taking absolute values in~\eqref{slope} gives $|p/q|=B(s)/A(s)$. Hence the Dehn-filling problem is reduced to proving that this quotient is bounded on the relevant real branches. For the estimates below, we record its explicit form. On the hyperbolic locus considered here, $f(s)>1$ and $g(s)\ge1$, so equations~\eqref{eq:f-definition} and~\eqref{eq:g-definition} give
			\begin{small}
			\[
			\begin{aligned}
			\frac{B(s)}{A(s)}
			&=\frac{2\operatorname{arccosh} g(s)}{\operatorname{arccosh} f(s)} =2\frac{\log \left( g(s)+\sqrt{g(s)^2-1} \right)}{\log \left( f(s)+\sqrt{f(s)^2-1} \right)}.
			\end{aligned}
			\]
			\end{small}
			In the next section, we prove this boundedness; throughout that argument, $A$, $B$, $f$, and $g$ denote the corresponding functions on the branch under consideration.

	\section{Boundedness of the slope function}
	\subsection{Endpoint estimates and proof}
		We first express $f$ and $g$ through the common variable $z=t+t^{-1}$ and then analyze the boundary regimes of the resulting algebraic branches. The first input is Riley's degree estimate.
	\begin{lemma}[Riley~\cite{Riley}, Lemma~2]
		Expand the entries $w_{ij}$ of $W$ as
		\[
		W=\left( \begin{matrix}
			\sum_{\nu\ge0}{A_{\nu}\left( t \right) s^{\nu}}&		\sum_{\nu\ge0}{B_{\nu}\left( t \right) s^{\nu}}\\
			\sum_{\nu\ge0}{C_{\nu}\left( t \right) s^{\nu}}&		\sum_{\nu\ge0}{D_{\nu}\left( t \right) s^{\nu}}\\
		\end{matrix} \right)
		\]
		where $A_{\nu}(t),\ldots,D_{\nu}(t)\in\mathbb{Z}[t,t^{-1}]$. Write $ m=(\alpha-1)/2 $. Then
		\[
		\deg_s\left( w_{11} \right) =\deg_s\left( w_{21} \right) =m=1+\deg_s\left( w_{12} \right) =1+\deg_s\left( w_{22} \right).
		\]
		Write $ \eta =(1-\varepsilon _1)/2 $, and recall $ \sigma =\sum_{j=1}^{\alpha -1}{\varepsilon _j} $. Then the coefficients of $ s^{m} $ in $ w_{11} $ and $ w_{21} $ are
		\[
		A_m\left( t \right) =\left( -1 \right) ^mt^{\sigma/2}, C_m\left( t \right) =\left( -1 \right) ^m\varepsilon _1t^{\sigma/2 +\eta}.
		\]
		Let $e(t,s)$ be any one of $w_{11}$, $w_{12}$, $w_{21}$, $w_{22}$. Following Riley's notation, write its positive $s$-degree part as
		\[
		e\left( t,s \right)-e\left( t,0 \right)=\sum_{\nu=1}^n{s^{\nu}}\sum_{l=-L}^L{e_{l\nu}t^l},\qquad e_{l\nu}\in \mathbb{Z}
		\]
		where $ n=m $ or $ m-1 $. Let $ k=k(e) $ be the largest index such that $ e_{kn}\ne 0 $. Then

		$ e_{l\nu}=0 $ if $ l+\nu >n+k $, and $ k\left( w_{11} \right) \geqslant k\left( w_{12} \right) $ and $ k\left( w_{21} \right) \geqslant k\left( w_{22} \right) $.
	\end{lemma}

		This degree estimate identifies the leading $s$-term of the normalized Riley polynomial.
	\begin{corollary}\label{cor:highest-s-term}
		The degree of $ \varPhi(t,s) $ as a polynomial in $ s $ is $ m $, and the coefficient of $ s^m $ is the constant Laurent polynomial $ (-1)^m $. Equivalently, the $ t $-constant coefficient of $ \varPhi(t,s) $ contains the term $ (-1)^m s^m $.
	\end{corollary}
	\begin{proof}
		By Riley's degree estimate, $ \deg_s\left( w_{11} \right) = m $ and $ \deg_s\left( w_{12} \right) =m-1 $. Hence the $ s^m $-term of
			$
		\varPhi(t,s)=t^{-\sigma/2} \left[w_{11}+(1-t)w_{12}\right]
		$
			comes only from $ t^{-\sigma/2}w_{11} $. Since the coefficient of $ s^m $ in $ w_{11} $ is $ A_{m}\left(t\right)= (-1)^{m}t^{\sigma/2} $, the coefficient of $ s^m $ in $ \varPhi(t,s) $ is $ t^{-\sigma/2}A_{m}\left(t\right)= (-1)^{m} $. Therefore $ \deg_s(\varPhi)= m $, and the $ t $-constant term contains the monomial $ (-1)^{m}s^m $.
	\end{proof}
	Applying Proposition~\ref{P1} and Lemma~\ref{lem:symmetric-laurent} with $z=t+t^{-1}$, write the Riley equation as
	\begin{equation}\label{eq:riley-z-polynomial}
		\varPhi \left( t,s \right)=E_r(s)z^r+E_{r-1}(s)z^{r-1}+\cdots +E_0(s)=0.
	\end{equation}
	Set
	\[
	L(z,s)=E_r(s)z^r+E_{r-1}(s)z^{r-1}+\cdots+E_0(s),
	\]
	so that $L(z,s)=0$ is the Riley equation in the symmetric variable $z$.
	This degree information passes to the symmetric $z$-polynomial in the following form.
	\begin{corollary}\label{cor:E-degree}
		With the notation of~\eqref{eq:riley-z-polynomial}, one has
		\[
		\deg E_0(s)=m,\qquad \deg E_i(s)<m \quad \text{for every } i\ge 1.
		\]
		In particular, $ \deg E_0(s)>\deg E_i(s) $ for every $ i\ge 1 $.
	\end{corollary}
	\begin{proof}
		By Corollary~\ref{cor:highest-s-term}, the coefficient of $s^m$ in $\varPhi(t,s)$ is the constant Laurent polynomial $(-1)^m$. Let $c_i$ be the coefficient of $s^m$ in $E_i(s)$. Then the coefficient of $s^m$ in the right-hand side of~\eqref{eq:riley-z-polynomial} is
		\[
		c_rz^r+c_{r-1}z^{r-1}+\cdots+c_0.
		\]
		The polynomials $1,z,\ldots,z^r$ are linearly independent as Laurent polynomials in $t$, since each $z^j=(t+t^{-1})^j$ has highest $t$-degree $j$. Hence this expression can equal the constant $(-1)^m$ only if $c_i=0$ for every $i\ge 1$ and $c_0=(-1)^m$. Therefore $E_0(s)$ has degree $m$, while each $E_i(s)$ with $i\ge 1$ has degree strictly less than $m$.
	\end{proof}

		The same symmetric variable also controls the longitude trace. Since $\operatorname{tr}\Lambda(t,s)$ is a Laurent polynomial in $t$ with coefficients in $\mathbb{Z}[s]$ and is symmetric by Proposition~\ref{prop:lambda-trace-symmetry}, the half-trace can be written in the form
		\[
		\frac{\operatorname{tr}\Lambda(t,s)}{2}=Q(t+t^{-1},s),\qquad Q(z,s)\in \frac12\mathbb{Z}[z,s].
		\]
	Let $\mathcal A$ be the semialgebraic set of real pairs $(s,z)$ satisfying
		\[
		L(z,s)=0,\qquad z>2,\qquad |Q(z,s)|\ge 1.
		\]
		Every hyperbolic real representation considered above gives a point of $\mathcal A$: the equation $L=0$ is Riley's equation, $z>2$ is the hyperbolic meridian condition, and $|Q|\ge1$ follows because the longitude commutes with the hyperbolic meridian. Let $\Omega:=\operatorname{pr}_s(\mathcal A)$. Since $\mathcal A$ is semialgebraic, choose a finite semialgebraic decomposition compatible with the projection to the $s$-line~\cite[Theorem~2.3.1 and Definition~2.3.4]{BCR}.

	\smallskip
	\noindent\emph{Reduction to branch intervals.}
	The non-vertical one-dimensional cells give open intervals, each carrying continuous algebraic branches $z(s)$, $f(s)=z(s)/2$, and $g(s)$; denote this finite family of intervals by $\mathcal I$. On such an interval, $|Q(z(s),s)|=g(s)\ge1$, so $Q(z(s),s)$ never vanishes and has a constant sign. After multiplying $Q$ by this sign on the interval, we may write, for some fixed nonnegative integer $\kappa$ depending on that interval,
	\begin{equation}\label{eq:g-z-polynomial}
	g(s)=H_\kappa(s)z(s)^\kappa+H_{\kappa-1}(s)z(s)^{\kappa-1}+\cdots+H_1(s)z(s)+H_0(s).
	\end{equation}
	Here each $H_i(s)$ lies in $\mathbb{Q}[s]$. Thus, on each fixed interval, $f$ and $g$ are controlled by the same algebraic branch $z=z(s)$. The chosen cell decomposition separates the branches: self-intersections and branch changes are confined to finitely many endpoints or vertical cells. Hence all endpoint estimates below are one-sided estimates along a fixed branch.

	We now dispose of the exceptional cells. Zero-dimensional cells contribute only finitely many values with $A>0$. A vertical cell lies over a fixed value $s=s_0$, with $z$ varying in a connected subset of $(2,\infty)$. On such a cell, $f=z/2$ and, after choosing the sign of $Q$ on the cell, $g=Q(z,s_0)$ is a rational-coefficient polynomial in $z$ with $g\ge1$. Therefore $B/A$ is bounded on compact $z$-subintervals. If $z\to\infty$, polynomial growth of $g$ gives the same logarithmic bound as in the unbounded-end case below.

	It remains only to consider a vertical-cell end with $z\to2^+$. Then $(2,s_0)$ lies in the cell closure; at $t=1$ the limiting meridian is nontrivial parabolic, and the limiting longitude commutes with it, so its trace is $\pm2$. After the sign choice for $Q$, this gives $Q(2,s_0)=1$. Hence $g-1=Q(z,s_0)-Q(2,s_0)=O(z-2)=O(f-1)$, and $\operatorname{arccosh}x\sim\sqrt{2(x-1)}$ gives boundedness of $B/A$. It remains to bound $B/A$ on the intervals in $\mathcal I$.
			\begin{lemma}\label{lem:boundary-behavior}
				Let $I$ be a branch interval in $\mathcal I$, and use the corresponding continuous branches of $f$ and $g$ on $I$.
				\begin{enumerate}[label=\textup{(\Roman*)}]
					\item If $I$ is unbounded and $|s|\to \infty$ in $I$, then $f(s)\to \infty$. Consequently, $A(s)\to \infty$.
					\item Let $s_0\in \overline{I}\setminus I$ be finite. If $f(s)\to 1$ as $s\to s_0$ in $I$, then $g(s)\to 1$. Consequently, $B(s)\to 0$ whenever $A(s)\to 0$.
					\item Let $s_0\in \overline{I}\setminus I$ be finite. If $g(s)\to \infty$ as $s\to s_0$ in $I$, then $f(s)\to \infty$. Consequently, $A(s)\to \infty$ whenever $B(s)\to \infty$.
				\end{enumerate}
			\end{lemma}
			\begin{proof}
				\noindent(I) Write $z=z(s)$ for the corresponding root of~\eqref{eq:riley-z-polynomial}, and let $|s|\to\infty$ in $I$. By Corollary~\ref{cor:E-degree}, $\deg E_0=m$ and $\deg E_i<m$ for every $i\ge 1$. If $z$ were bounded along a sequence with $|s|\to\infty$, then
			\[
			E_r(s)z^r+\cdots+E_1(s)z=O(|s|^{m-1}),
			\]
			whereas
			\[
			E_0(s)=(-1)^m s^m+O(|s|^{m-1}).
			\]
			Hence the left-hand side of~\eqref{eq:riley-z-polynomial} would be $(-1)^m s^m+O(|s|^{m-1})$, which cannot vanish for large $|s|$. Thus $|z|\to\infty$ along the branch. Since $z=2f(s)>2$ on the hyperbolic branch, we get $z\to+\infty$. Therefore $f(s)=z/2\to\infty$, and consequently $A(s)\to\infty$.

				\noindent(II) Since $f(s)=(t+t^{-1})/2$ on the chosen branch and $t>0$ in the hyperbolic case, the condition $f(s)\to 1$ implies $t\to 1$. The diagonalization used above is a conjugation of the whole representation for $t\ne1$, so it does not change the longitude trace. We may therefore compute this trace in the non-diagonal unimodular normal form, which has a well-defined parabolic limit at $t=1$. In this form the meridian tends to the parabolic matrix
			\[
			x\rightarrow \left( \begin{smallmatrix}
				1&		1\\
				0&		1
			\end{smallmatrix} \right).
			\]
			Because $s_0$ is finite, the other generator also has a finite limit. Hence the longitude $ \lambda=x^{-2\sigma}w^*w $, being a fixed word in $x^{\pm1}$ and $y^{\pm1}$, has a finite limit. Moreover, $x\lambda=\lambda x$ for every $s$, so the limiting longitude commutes with this nontrivial unipotent matrix. Its centralizer in $\mathrm{SL}(2,\mathbb{C})$ consists of the matrices
			\[
			\left( \begin{smallmatrix}
				\epsilon&		c\\
				0&		\epsilon
			\end{smallmatrix} \right),\qquad \epsilon=\pm 1,\quad c\in\mathbb{C}.
			\]
			Thus $|\operatorname{tr}\rho(\lambda)|/2\rightarrow 1$, and consequently $g(s)\rightarrow 1$. Since $f(s)=\cosh(2\tilde a(s))$ and $g(s)=\cosh \tilde b(s)$ by definition, this also gives the stated implication $A(s)\to 0 \Rightarrow B(s)\to 0$.

			\noindent(III) In~\eqref{eq:g-z-polynomial}, the coefficients $H_i(s)$ are rational-coefficient polynomials, hence remain bounded as $s\to s_0$. Since $g(s)\to\infty$, the branch value $z$ cannot have a bounded subsequence. Therefore $|z|\to\infty$. On the hyperbolic branch, $z=2f(s)>2$, so $z\to+\infty$, and hence $f(s)\to\infty$.
		\end{proof}

		The small-parameter boundary also requires the following derivative estimate.
		\begin{lemma}\label{lem:differentiability}
			Let $I$ be a branch interval in $\mathcal I$, and let $s_0\in \overline{I}\setminus I$ be finite. Suppose $f(s)\to 1$ as $s\to s_0$ in $I$, so that we are in case~(II) of Lemma~\ref{lem:boundary-behavior}. Extend $f$ and $g$ to $s_0$ by their limiting values. If $f$ has a one-sided derivative at $s_0$ along $I$, then $g$ also has a one-sided derivative at $s_0$ along $I$. Moreover, $f'(s_0)\neq 0$.
		\end{lemma}
		\begin{proof}
			Recall that $L(z(s),s)=0$ along the chosen branch. Since $z(s)=2f(s)$ and $f$ is one-sided differentiable at $s_0$, the branch $z(s)$ is one-sided differentiable at $s_0$ and $z(s)\to 2$. Taking the one-sided derivative at $s_0$ in the identity $L(z(s),s)=0$ gives
			\[
			L_z(2,s_0)z'(s_0)+L_s(2,s_0)=0.
			\]
			Let $R(s)=L(2,s)=\varPhi(1,s)$, where the equality follows by setting $t=1$ and hence $z=t+t^{-1}=2$. Since $f(s)\to 1$, we have $R(s_0)=L(2,s_0)=0$. Moreover,
			\[
			R'(s_0)=L_s(2,s_0).
			\]
			By Lemma~\ref{P2}, the polynomial $R=\varPhi(1,s)$ has no repeated factors in $\mathbb{Z}[s]$, which means that its zeros are simple. Since $R(s_0)=0$, we have $R'(s_0)\neq 0$. Hence $L_s(2,s_0)\neq 0$.

			If $L_z(2,s_0)=0$, then the differentiated identity at $s_0$ would imply $L_s(2,s_0)=0$, a contradiction. Therefore $L_z(2,s_0)\neq 0$, and
			\[
			z'(s_0)=-\frac{L_s(2,s_0)}{L_z(2,s_0)}\neq 0.
			\]
			Thus $f'(s_0)=z'(s_0)/2\neq 0$.

			Finally, equation~\eqref{eq:g-z-polynomial} expresses $g(s)$ as a polynomial in $z(s)$ with polynomial coefficients in $s$. Since $s_0$ is finite and $z(s)$ is one-sided differentiable at $s_0$, this expression is one-sided differentiable at $s_0$. Hence $g$ is one-sided differentiable at $s_0$.
		\end{proof}

		The next elementary lemma handles endpoints where both translation lengths tend to zero.
	\begin{lemma}\label{lem:small-parameter-ratio}
		Let $I$ be a branch interval in $\mathcal I$, and let $s_0\in\overline{I}\setminus I$ be finite. Suppose that $A(s),B(s)\to 0$ as $s\to s_0$ within $I$. If $\frac{g(s)-1}{f(s)-1}$ is bounded near $s_0$ along $I$, then $\frac{B(s)}{A(s)}$ is also bounded near $s_0$ along $I$.
	\end{lemma}
	\begin{proof}
		By Lemma~\ref{lem:boundary-behavior}(II), we have $f(s)\to 1$ and $g(s)\to 1$, with $f(s)>1$ and $g(s)\ge1$ on the branch. The standard expansion $\operatorname{arccosh}(x)\sim \sqrt{2(x-1)}$ as $x\to 1^+$ gives constants $c_1,c_2>0$ and a neighborhood of $s_0$ such that
		\[
		\operatorname{arccosh} f(s)\ge c_1\sqrt{f(s)-1},
		\qquad
		\operatorname{arccosh} g(s)\le c_2\sqrt{g(s)-1}
		\]
		there. Hence
		\[
		\frac{B(s)}{A(s)}=2\frac{\operatorname{arccosh} g(s)}{\operatorname{arccosh} f(s)}
		\le \frac{2c_2}{c_1}\sqrt{\frac{g(s)-1}{f(s)-1}},
		\]
		and boundedness of $\frac{g(s)-1}{f(s)-1}$ implies boundedness of $\frac{B(s)}{A(s)}$ near $s_0$.
	\end{proof}

		With these preparations in place, we prove the main boundedness theorem.
	\begin{proof}[Proof of Theorem~\ref{thm:main}]
		By~\eqref{slope}, it suffices to prove that $B(s)/A(s)$ is bounded on the admissible real branches. Fix a branch interval $I$ in $\mathcal I$. On every compact subinterval of $I$, the functions $A(s)$ and $B(s)$ are continuous and $A(s)>0$, so $B(s)/A(s)$ is bounded. It remains to prove boundedness near each finite endpoint of $I$ and along each unbounded end of $I$; otherwise any unbounded sequence for the quotient would have a subsequence of one of these types. Since $f$ and $g$ are semialgebraic on the one-dimensional cell, they have one-sided limits in $[1,\infty]$ at every such end; hence $A$ and $B$ have one-sided limits in $[0,\infty]$ as well. The cases below therefore exhaust the possible endpoint behaviors.

		\noindent\textup{(I) Unbounded ends.} Suppose $I$ is unbounded and $|s|\to\infty$ along an end of $I$. By Lemma~\ref{lem:boundary-behavior}(I), $f(s)\to\infty$ and hence $A(s)\to\infty$. If $B(s)$ remains bounded, then $B(s)/A(s)\to 0$. Thus the only nontrivial case is $B(s)\to\infty$. In that case $g(s)=\cosh B(s)\to\infty$, and
		\[
		\frac{B(s)}{A(s)}=2\frac{\log\left(g+\sqrt{g^2-1}\right)}{\log\left(f+\sqrt{f^2-1}\right)}.
		\]
		Because $f$ and $g$ are semialgebraic functions on the chosen one-sided branch, their germs at the end are represented by algebraic Puiseux series~\cite[Example~7.3.6 and Proposition~8.1.12]{BCR}. Thus $f(s)=c|s|^{d_f}(1+o(1))$ for some $c>0$ and rational $d_f>0$, while $g(s)=O(|s|^{d_g})$ for some $d_g>0$. Hence
		\[
		\log\left(f+\sqrt{f^2-1}\right)=d_f\log|s|+O(1),
		\]
		where the denominator is eventually positive and unbounded. Since $g\ge1$ and $g=O(|s|^{d_g})$, the elementary estimate $\operatorname{arccosh}x=\log(x+\sqrt{x^2-1})\le \log(2x)$ for $x\ge1$ gives
		\[
		\log\left(g+\sqrt{g^2-1}\right)=O(\log|s|).
		\]
		Therefore
		\[
		\limsup_{|s|\to\infty}\frac{B(s)}{A(s)}<\infty.
		\]

		\noindent\textup{(II) Finite endpoints with large parameters.} Let $s_0\in\overline{I}\setminus I$ be finite, and suppose $A(s),B(s)\to\infty$ as $s\to s_0$ within $I$. Then $g(s)=\cosh B(s)\to\infty$, so Lemma~\ref{lem:boundary-behavior}(III) gives $f(s)\to\infty$ as well. We reduce this finite endpoint to the preceding unbounded-end argument by setting
		\[
		u=(s-s_0)^{-1}
		\]
		and writing the corresponding branches as $\check f(u)=f(s_0+u^{-1})$ and $\check g(u)=g(s_0+u^{-1})$. These are still semialgebraic functions of $u$, and $\check f(u),\check g(u)\to\infty$ as $|u|\to\infty$ along the transformed one-sided branch. By the same Puiseux description of one-sided semialgebraic function germs~\cite[Example~7.3.6 and Proposition~8.1.12]{BCR}, $\check f(u)=c|u|^{d_f}(1+o(1))$ with $c>0$ and rational $d_f>0$. Similarly $\check g(u)=O(|u|^{d_g})$ for some $d_g>0$. Thus the denominator grows as $d_f\log|u|+O(1)$ and the numerator is $O(\log|u|)$, exactly as in (I), so
		\[
		\frac{B}{A}=2\frac{\log\left(\check g+\sqrt{\check g^2-1}\right)}{\log\left(\check f+\sqrt{\check f^2-1}\right)}=O(1)
		\]
		as $|u|\to\infty$. Hence $\frac{B(s)}{A(s)}$ is bounded near $s_0$.

		\noindent\textup{(III) Finite endpoints with small parameters.} Let $s_0\in\overline{I}\setminus I$ be finite, and suppose $A(s),B(s)\to 0$ as $s\to s_0$ within $I$. First suppose that $f$ has a one-sided derivative at $s_0$ along $I$. Lemma~\ref{lem:differentiability} gives $g'(s_0)$ finite and $f'(s_0)\neq 0$, so
		\[
		\frac{g(s)-1}{f(s)-1}\longrightarrow \frac{g'(s_0)}{f'(s_0)}
		\qquad \text{as } s\to s_0.
		\]
		Lemma~\ref{lem:small-parameter-ratio} then shows that $\frac{B(s)}{A(s)}$ is bounded near $s_0$.

			Now suppose that $f$ has no one-sided derivative at $s_0$ along $I$. Since $f(s)-1$ is semialgebraic and tends to $0$, the one-sided Puiseux description of semialgebraic function germs gives a local expansion in $\tau=|s-s_0|$~\cite[Example~7.3.6 and Proposition~8.1.12]{BCR}:
			\[
			f(s)-1=c\tau^r+o\left(\tau^r\right)
			\]
			along the chosen branch, with $c\ne 0$ and rational $r>0$. Since $I$ approaches $s_0$ from one side, $s-s_0=\epsilon\tau$ for a fixed sign $\epsilon\in\{\pm1\}$. If $r\ge 1$, then $f(s)-1$ would have a finite one-sided derivative at $s_0$ along $I$ (when $r=1$ the derivative is $\epsilon c$, and when $r>1$ it is $0$), contrary to the present assumption. Hence $0<r<1$, and therefore
		\begin{equation}\label{eq:branch-quotient-diverges}
			\lim_{s\to s_0}\left|\frac{f-1}{s-s_0}\right|=\infty.
		\end{equation}
		We next show that this singular behavior of $f$ does not make $(g-1)/(f-1)$ unbounded. Using~\eqref{eq:g-z-polynomial} and $z(s)=2f(s)$, write
		\[
		g(s)=G_\kappa(s)f^\kappa+G_{\kappa-1}(s)f^{\kappa-1}+\cdots+G_1(s)f+G_0(s),
		\qquad G_i(s)=2^iH_i(s).
		\]
		Since $f(s_0)=g(s_0)=1$, we have $\sum_{i=0}^{\kappa}G_i(s_0)=1$. Therefore, after adding and subtracting $\sum_iG_i(s)$,
		\[
		\frac{g-1}{f-1}=\sum_{i=1}^{\kappa}G_i(s)\frac{f^i-1}{f-1}+\frac{\sum_{i=0}^{\kappa}G_i(s)-1}{f-1}.
		\]
		For each $i\ge 1$,
		\[
		\frac{f^i-1}{f-1}=1+f+\cdots+f^{i-1}\longrightarrow i
		\qquad \text{as } s\to s_0,
		\]
		so the sum has a finite limit. The numerator of the remaining term vanishes at $s_0$ and is therefore $O(|s-s_0|)$, while the denominator dominates $|s-s_0|$ by~\eqref{eq:branch-quotient-diverges}; hence that remaining term tends to $0$. Thus $\frac{g-1}{f-1}$ has a finite limit, and Lemma~\ref{lem:small-parameter-ratio} again gives boundedness of $\frac{B(s)}{A(s)}$ near $s_0$.

		It remains only to record that the preceding estimates control all other endpoint behaviors. If $A(s)\to\infty$ while $B(s)$ stays bounded, then $B(s)/A(s)\to 0$. If $B(s)\to 0$ while $A(s)$ stays bounded away from $0$, then again $B(s)/A(s)\to 0$. If both $A(s)$ and $B(s)$ remain bounded and $A(s)$ stays bounded away from $0$, then the ratio is bounded. Finally, Lemma~\ref{lem:boundary-behavior}(II) rules out the case $A(s)\to 0$ while $B(s)$ stays away from $0$, and Lemma~\ref{lem:boundary-behavior}(III) reduces the case $B(s)\to\infty$ to the large-parameter case above.

		Thus $\frac{B(s)}{A(s)}$ is bounded on $I$. Since $I$ was arbitrary, the ratio is bounded on every branch interval in $\mathcal I$. Combining these bounds with the finitely many vertical and zero-dimensional contributions from the reduction above, we obtain a common finite bound over all admissible cells. By the necessary condition~\eqref{slope}, the corresponding surgery slopes form a bounded set.

		It remains to justify the final $\mathrm{SL}(2,\mathbb{R})$ statement. Let $\rho_{\mathrm{SL}}$ be a non-abelian $\mathrm{SL}(2,\mathbb{R})$ representation with hyperbolic meridian image, and put $M=\rho_{\mathrm{SL}}(\mu)$. Its projection to $\mathrm{PSL}(2,\mathbb{R})$ is hyperbolic. We claim that this projection is still non-abelian. If not, then for every matrix $N$ in the image of $\rho_{\mathrm{SL}}$ one has $MN=\pm NM$. The minus sign cannot occur: it would imply $NMN^{-1}=-M$, hence $\operatorname{tr}M=0$, contradicting the hyperbolicity of $M$. Thus every $N$ commutes with $M$, so the image lies in the centralizer of a hyperbolic matrix, which is abelian. This contradicts the non-abelianity of $\rho_{\mathrm{SL}}$. Hence any such $\mathrm{SL}(2,\mathbb{R})$ representation projects to one of the $\mathrm{PSL}(2,\mathbb{R})$ representations already excluded above.
	\end{proof}

		We record the limited effectivity provided by the proof before turning to examples.
\begin{remark}\label{rem:effective-bound}
The argument gives, in principle, an effective but generally coarse constant in Theorem~\ref{thm:main}. For a fixed knot, it is obtained from the Riley polynomial, the longitude trace, a semialgebraic decomposition of the admissible set $\mathcal A\subset\mathbb{R}^2$, and Puiseux estimates at the finitely many endpoints and at infinity. We do not attempt to optimize this constant.
\end{remark}

	\subsection{Computational illustrations}
The following examples illustrate the parameter set $\Omega$ and the endpoint estimates for $B/A$; they are not intended to give sharp global maxima.

\begin{example}[The knot $K(5,3)$]
For $K(5,3)$ one has
\[
\varPhi(t,s)=(s^2+3s+3)-(s+1)(t+t^{-1}).
\]
Thus, with $z=t+t^{-1}$,
\[
 z(s)=\frac{s^2+3s+3}{s+1},
 \qquad
 f(s)=\frac{z(s)}2=\frac{s^2+3s+3}{2(s+1)}.
\]
The hyperbolic condition $z>2$ gives
\[
\Omega=(-1,\infty).
\]
A direct trace calculation for the longitude gives
\[
 g(s)=\cosh B(s)=\frac{s^4+5s^3+7s^2+4s+2}{2(s+1)^2}.
\]
Therefore
\[
\frac{B(s)}{A(s)}
=2\frac{\operatorname{arccosh} g(s)}{\operatorname{arccosh} f(s)}.
\]
At the unbounded end, $f(s)\sim s/2$ and $g(s)\sim s^2/2$, so
\[
\lim_{s\to+\infty}\frac{B(s)}{A(s)}=4.
\]
At the finite endpoint, putting $c=(s+1)^{-1}$ gives
\[
 f(s)=\frac12(c+c^{-1}+1),
 \qquad
 g(s)=\frac12(c^2+c-2+c^{-1}+c^{-2}),
\]
and hence
\[
\lim_{s\to-1^+}\frac{B(s)}{A(s)}=4.
\]
On every compact subinterval of $(-1,\infty)$, zeros of $B$ do not affect boundedness of the quotient. Thus this example illustrates the unbounded-end case and the finite large-parameter endpoint case of the proof.
\end{example}

The next example has two projected components of $\Omega$, and so illustrates how the endpoint analysis behaves on different branches.
\begin{example}[The knot $K(9,7)$]
For $K(9,7)$ the symmetric Riley equation is linear in $z=t+t^{-1}$:
\[
\varPhi(t,s)=\bigl(s^4+7s^3+17s^2+16s+5\bigr)-\left(s^3+5s^2+7s+2\right)z.
\]
Thus
\[
 z(s)=\frac{s^4+7s^3+17s^2+16s+5}{s^3+5s^2+7s+2},
 \qquad
 f(s)=\frac{z(s)}2.
\]
Since
\[
 s^3+5s^2+7s+2=(s+2)(s^2+3s+1),
\]
one checks that $z>2$ precisely on
\[
\Omega=\left(\frac{-3-\sqrt5}{2},-2\right)\cup
\left(\frac{-3+\sqrt5}{2},\infty\right).
\]
The longitude trace gives
\[
 g(s)=\cosh B(s)=
 \frac{s^8+13s^7+67s^6+174s^5+239s^4+168s^3+58s^2+12s+2}
 {2(s^2+3s+1)^2}.
\]
As $s\to+\infty$, we have $f(s)\sim s/2$ and $g(s)\sim s^4/2$, and therefore
\[
\lim_{s\to+\infty}\frac{B(s)}{A(s)}=8.
\]
At the finite endpoint $s_+=(-3+\sqrt5)/2$ of the second component, one has
\[
 f(s)\sim \frac{1/4-\sqrt5/20}{s-s_+},
 \qquad
 g(s)\sim \frac{3/20-\sqrt5/20}{(s-s_+)^2},
\]
so
\[
\lim_{s\to s_+^+}\frac{B(s)}{A(s)}=4.
\]
On the first component, if $s_-=(-3-\sqrt5)/2$, then as $s\to s_-^+$,
\[
 f(s)\sim \frac{1/4+\sqrt5/20}{s-s_-},
 \qquad
 g(s)\sim \frac{3/20+\sqrt5/20}{(s-s_-)^2},
\]
again giving
\[
\lim_{s\to s_-^+}\frac{B(s)}{A(s)}=4.
\]
Finally, as $s\to-2^-$, $f(s)\to\infty$ while $g(s)\to1$, and hence
\[
\lim_{s\to-2^-}\frac{B(s)}{A(s)}=0.
\]
Thus this example displays two projected components of $\Omega$ and several endpoint behaviors that occur in the general proof.
\end{example}

	\subsection{Consequences and limitations}
\begin{corollary}\label{cor:univ-cover-method}
For a fixed two-bridge knot $K$, there is a constant $C(K)>0$ with the following property. If $q\neq0$, $|p/q|>C(K)$, and
\[
\widetilde{\rho}:\pi_1\bigl(S^3_{p/q}(K)\bigr)\longrightarrow \widetilde{\mathrm{PSL}}(2,\mathbb{R})
\]
is any representation, then the projected representation
\[
\pi\circ\widetilde{\rho}:\pi_1\bigl(S^3_{p/q}(K)\bigr)\longrightarrow \mathrm{PSL}(2,\mathbb{R})
\]
is not both hyperbolic and non-abelian. Equivalently, for every such large rational slope, each projected representation is either abelian or has non-hyperbolic meridian image.
\end{corollary}
\begin{proof}
Let
\[
\pi:\widetilde{\mathrm{PSL}}(2,\mathbb{R})\longrightarrow \mathrm{PSL}(2,\mathbb{R})
\]
be the covering projection. This step uses only the canonical projection from the universal cover; it does not require a separate lifting argument.
Let $C(K)$ be the constant from Theorem~\ref{thm:main}. If $|p/q|>C(K)$ and $\pi\circ \widetilde{\rho}$ were non-abelian and hyperbolic, then $S^3_{p/q}(K)$ would admit a non-abelian $\mathrm{PSL}(2,\mathbb{R})$ representation with hyperbolic meridian image, contradicting Theorem~\ref{thm:main}. Therefore the projection must be either abelian or non-hyperbolic.
\end{proof}

The following remarks clarify what the theorem and corollary do not imply.
\begin{remark}
Corollary~\ref{cor:univ-cover-method} does not rule out left-orderability for large surgeries. Rather, it shows that such an argument cannot rely on projected representations that are both non-abelian and hyperbolic.
After projection to $\mathrm{PSL}(2,\mathbb{R})$, any relevant representation must either have abelian image or have elliptic, parabolic, or otherwise non-hyperbolic meridian image.
\end{remark}

\begin{remark}
Theorem~\ref{thm:main} by itself is not a vanishing theorem for Seifert volume. In the formulation of Seifert volume via $\mathrm{PSL}(2,\mathbb{R})$ representations and Godbillon--Vey invariants~\cite{Volume,SU}, the theorem eliminates only non-abelian representations with hyperbolic meridian image for large slopes. Abelian representations contribute trivially in the standard representation-volume settings used here~\cite[Proposition~6.1]{SU}~\cite[Lemma~3.6]{DLW}.
Thus any possible large-slope nonzero contribution would have to come from non-abelian $\mathrm{PSL}(2,\mathbb{R})$ representations with non-hyperbolic meridian image. Additional peripheral input is needed to turn this observation into a Seifert-volume vanishing theorem.
\end{remark}

\section*{Acknowledgements}
The second author was supported by NSFC (No.~11901413). He thanks Professor Yi Liu for his hospitality during the second author's visit to BICMR in 2022.


\begin{thebibliography}{99}

		\bibitem{BCR}
		J. Bochnak, M. Coste, and M.-F. Roy, \emph{Real Algebraic Geometry}, Ergeb. Math. Grenzgeb. (3), vol.~36, Springer-Verlag, Berlin, 1998. DOI:10.1007/978-3-662-03718-8.

		\bibitem{Orderable}
		S. Boyer, D. Rolfsen, and B. Wiest, Orderable 3-manifold groups, Ann. Inst. Fourier (Grenoble) \textbf{55} (2005), no.~1, 243--288. DOI:10.5802/aif.2098.

		\bibitem{Volume}
		R. Brooks and W. Goldman, Volumes in Seifert space, Duke Math. J. \textbf{51} (1984), no.~3, 529--545. DOI:10.1215/S0012-7094-84-05126-3.

		\bibitem{CullerShalen}
		M. Culler and P. B. Shalen, Varieties of group representations and splittings of 3-manifolds, Ann. of Math. (2) \textbf{117} (1983), no.~1, 109--146. DOI:10.2307/2006973.

		\bibitem{DLW}
		P. Derbez, Y. Liu, and S. Wang, Chern--Simons theory, surface separability, and volumes of 3-manifolds, J. Topol. \textbf{8} (2015), no.~4, 933--974. DOI:10.1112/jtopol/jtv023.

		\bibitem{Gao}
		X. Gao, Slope of orderable Dehn filling of two-bridge knots, J. Knot Theory Ramifications \textbf{31} (2022), no.~1, Paper No.~2250006, 24 pp. DOI:10.1142/S0218216522500067.

		\bibitem{HakaTeraGenusOne}
		R. Hakamata and M. Teragaito, Left-orderable fundamental groups and Dehn surgery on genus one 2-bridge knots, Algebr. Geom. Topol. \textbf{14} (2014), no.~4, 2125--2148. DOI:10.2140/agt.2014.14.2125.

		\bibitem{SU}
		V. T. Khoi, A cut-and-paste method for computing the Seifert volumes, Math. Ann. \textbf{326} (2003), no.~4, 759--801. DOI:10.1007/s00208-003-0438-5.

		\bibitem{Tran2}
		V. T. Khoi, M. Teragaito, and A. T. Tran, Left orderable surgeries of double twist knots II, Canad. Math. Bull. \textbf{64} (2021), no.~3, 624--637. DOI:10.4153/S0008439520000703.

		\bibitem{Riley2}
		R. Riley, Parabolic representations of knot groups, I, Proc. London Math. Soc. (3) \textbf{24} (1972), no.~2, 217--242. DOI:10.1112/plms/s3-24.2.217.

		\bibitem{Riley}
		R. Riley, Nonabelian representations of 2-bridge knot groups, Quart. J. Math. Oxford Ser. (2) \textbf{35} (1984), no.~138, 191--208. DOI:10.1093/qmath/35.2.191.

		\bibitem{TeragaitoTwist}
		M. Teragaito, Left-orderability and exceptional Dehn surgery on twist knots, Canad. Math. Bull. \textbf{56} (2013), no.~4, 850--859. DOI:10.4153/CMB-2012-011-0.

		\bibitem{Thakar2023}
		O. Thakar, Left-orderable surgeries on the knot $6_2$ via hyperbolic $\widetilde{\mathrm{PSL}}(2,\mathbb{R})$-representations, arXiv:2307.00107, 2023.

		\bibitem{Tran}
		A. T. Tran, Left orderable surgeries of double twist knots, J. Math. Soc. Japan \textbf{73} (2021), no.~3, 753--765. DOI:10.2969/jmsj/84058405.

	\end{thebibliography}
\end{document}